\documentclass[a4paper]{amsart}
\usepackage{amsthm,amsmath,amssymb,mathrsfs}
\usepackage[latin1]{inputenc}
 \newtheorem{thm}{Theorem}
 \newtheorem{prop}[thm]{Proposition}
 \newtheorem{lemma}[thm]{Lemma}
 \newtheorem{cor}[thm]{Corollary}

 \theoremstyle{definition}
 \newtheorem{definition}[thm]{Definition}

 \theoremstyle{remark}
 \newtheorem{remark}[thm]{Remark}
\numberwithin{thm}{section}

\def\Spec{{\rm Spec}\,}
\def\Spf{{\rm Spf}}

\def\GL{{\rm GL}}

\def\GSp{{\rm GSp}}

\def\dom{{\rm dom}}

\def\dim{{\rm dim}\,}

\def\e{\epsilon}
\def\dl{(\!(}
\def\dr{)\!)}
\def\ll{[\![}
\def\rr{]\!]}

\def\N{{\mathcal N}}

\def\Z{{\mathbb Z}}
\def\F{{\mathbb F}}
\def\Q{{\mathbb Q}}

\def\O{{\mathcal O}}

\def\m{{\mathfrak m}}

\begin{document}

\begin{title}
{Central leaves in loop groups}
\end{title}
\author{Eva Viehmann and Han Wu}
\address{Technische Universit\"at M\"unchen\\Fakult\"at f\"ur Mathematik - M11 \\ Boltzmannstr.~3\\85748 Garching bei M\"unchen\\Germany}
\email{viehmann@ma.tum.de, wu@ma.tum.de}
\date{}
\thanks{The authors were partially supported by ERC starting grant 277889 ``Moduli spaces of local $G$-shtukas''.}

\begin{abstract}{This paper analyses the finer structure of  Newton strata in loop groups. These can be decomposed into so-called central leaves. We define them, and determine their global geometric structure. We then study the closure of central leaves, both by proving some general properties and by considering an illustrative example.}
\end{abstract}
\maketitle
\section{Introduction}\label{intro}

Let $F=\mathbb F_q\dl \e\dr $ be a local field of characteristic $p$, and let $\mathcal O_F$ be its ring of integers. Let $G$ be a connected reductive group over $\mathcal O_F$. Let $T\subset B$ be a maximal torus and a Borel subgroup of $G$. Let $L$ be the completion of the maximal unramified extension of $F$, let $\mathcal O_{L}$ be its ring of integers, and $k$ its residue field. Then $k$ is an algebraic closure of $\F_p$. Let $\sigma$ denote the Frobenius of $L$ over $F$ and also of $k$ over $\F_q$. 

Let $LG$ denote the loop group of $G$. Let $K$ be the subgroup of $LG$ with $K(R)=G(R\ll \e\rr)$. By abuse of notation we also write $K$ for $K(k)=G(\mathcal O_{L})$. We have the Cartan decomposition $G(L)=\coprod_{\mu\in X_*(T)_{\dom}}K\e^\mu K$ where the union runs over the set of cocharacters which are dominant with respect to $B$.

For $b\in G(L)$ denote by $[b]$ its $\sigma$-conjugacy class. As usual, $B(G)$ denotes the set of $\sigma$-conjugacy classes in $G(L)$. For $G=\GL_n$ the elements of $B(\GL_n)$ are classified by their Newton polygons. This classification is extended to all reductive groups by Kottwitz \cite{Kottwitz1}. The set $B(G)$ has a partial ordering $\preceq$ generalizing the natural ordering (of `lying above') on the set of Newton polygons. For a given $\mu\in X_*(T)_{\dom}$ let $B(G,\mu)$ denote the finite set of $[b]\in B(G)$ with $[b]\cap K\e^\mu K\neq \emptyset$. For $[b]\in B(G,\mu)$ let $\mathcal N_{[b],\mu}=[b]\cap K\e^{\mu} K$. It is called the Newton stratum of $[b]$ in $K\e^{\mu} K$.

The double coset $K\e^{\mu} K$ has a structure of an infinite-dimensional subscheme of the loop group of $G$, and each $\mathcal N_{[b],\mu}$ is a locally closed subscheme.
In order to study the geometry of a Newton stratum more closely, a natural tool is to decompose it into central leaves. Here the central leaf of an element $y\in \mathcal N_{[b],\mu}$ is the set $$\mathcal C_y=\{g^{-1}y\sigma(g)\mid g\in K\}.$$  In Section \ref{sec2} we show that this defines a smooth locally closed subscheme of $LG$ that is closed in $\mathcal N_{[b],\mu}$. 

This definition of central leaves is a natural group-theoretic analogue of the theory of central leaves in moduli spaces of abelian varieties considered by Oort in \cite{Oortfol}, and later also by several other people, for example Chai and Oort \cite{CO}, and Harashita \cite{Harashita}. 

Note that using the usual notion of dimension, all of these schemes (double cosets, Newton strata, and central leaves) are infinite-dimensional. However, using that they are invariant under a sufficiently small open subgroup of $K$ and that the corresponding quotient is finite-dimensional we define a notion of dimension for these schemes that is finite in all three cases, and allows to compare them. In Theorem \ref{thmdim} we compute the dimensions of central leaves $\mathcal C_y$. We show in particular that their dimension only depends on the class $[y]\in B(G)$, and neither on $\mu$, nor on the specific element within the Newton stratum. These dimensions also agree with the dimension of similarly defined central leaves in (finite-dimensional) deformation spaces of local $G$-shtukas, and of the corresponding leaves in moduli spaces of abelian varieties.

Another foundational question on the decomposition of $K\e^{\mu}K $ into central leaves is to determine the closure of a given central leaf. Central leaves are closed within their Newton stratum. However, only the basic Newton strata are closed within $K\e^{\mu} K$ (see \cite{grothconj} for a complete description of the closure in the case that $G$ is split, or Proposition \ref{thmfund} below which implies that non-basic Newton strata can never be closed). Surprisingly, the question for the closure in $K\e^{\mu} K$ of a given central leaf $\mathcal C_x\subset\mathcal N_{[b],\mu}$ seems to not have been studied before except for the very particular case explained in Proposition \ref{thmfund} below. Even in the context of central leaves in moduli spaces of abelian varieties one does not know more about these closures. Using essentially the same arguments, one can prove an analog of our results on closures also in this arithmetic context.

A first rough answer on the shape of the closure is the following.

\begin{remark} By definition the closure of a central leaf is invariant under $K$-$\sigma$-conjugation, and is thus again a union of central leaves. As every central leaf $\mathcal C$ in some $\mathcal N_{[b],\mu}$ is closed in that same Newton stratum by Corollary \ref{corclosed}, its closure consists of $\mathcal C$ and a union of central leaves in Newton strata for $[b']\prec [b]$.
\end{remark}

Lemma \ref{lem1} and Proposition \ref{thmfund} are two cases of closures of central leaves that can be easily shown or deduced from the literature. To state them we need some notation. For $\mu\in X_*(T)_{\dom}$ and $b\in G(L)$ with $[b]\in B(G,\mu)$ let $$X_{\mu}(b)=\{g\in LG/K\mid g^{-1}b\sigma(g)\in K\e^{\mu}K\}.$$ This defines a locally closed reduced subscheme of the affine Grassmannian $LG/K$ called the affine Deligne-Lusztig variety. It carries a natural action by multiplication on the left by the group $$J_b(F)=\{g\in G(L)\mid g^{-1}b\sigma(g)=b\}.$$

\begin{lemma}\label{lem1}
Assume that $[b]\in B(G,\mu)$ and $\dim X_{\mu}(b)=0$. Then $\mathcal N_{[b],\mu}$ consists of a single central leaf. In particular, the closure of this central leaf coincides with $\overline{\mathcal N_{[b],\mu}}$.
\end{lemma}
\begin{proof}
As $\dim X_{\mu}(b)=0$ and $X_{\mu}(b)$ is a closed subscheme of the affine Grassmannian that is locally of finite type, we have $X_{\mu}(b)=\pi_0(X_{\mu}(b))$. Let $X_{\preceq\mu}(b)=\bigcup_{\mu'\preceq\mu}X_{\mu'}(b)$. Then $X_{\mu}(b)$ is open in $X_{\preceq\mu}(b)$. As it is also a closed scheme, we have $\pi_0(X_{\mu}(b))\subset\pi_0(X_{\preceq\mu}(b))$. By \cite{Niecc}, Theorem 1.2, the group $J_b(F)$ acts transitively on $\pi_0(X_{\preceq\mu}(b))$. As it also stabilizes the non-empty subset $\pi_0(X_{\mu}(b))$, we have $\pi_0(X_{\mu}(b))=\pi_0(X_{\preceq\mu}(b))$. Thus $J_b(F)$ acts transitively on $X_{\mu}(b)$. In other words, the subset of $K\e^{\mu} K$ of elements of the form $g^{-1}b\sigma(g)$ with $gK\in X_{\mu}(b)$ is a single $K$-$\sigma$-conjugacy class. This shows the first assertion. The second is a direct consequence. 
\end{proof}
Note that the converse of this lemma also holds: If a Newton stratum consists of a single central leaf, then the dimension of the corresponding affine Deligne-Lusztig variety is $0$. This lemma applies in particular to the generic Newton stratum in $K\e^{\mu}K$, also called the $\mu$-ordinary locus. For a different proof of the statement in that case compare \cite{Wortmann}, Proposition 6.9. We will also use this lemma for several specific examples in Section \ref{sec3}.

Let $I$ be the Iwahori subgroup in $K$ whose image under the projection to $G(k)$ is $B$. We have the Cartan decomposition $G(L)=\coprod_{x\in\widetilde W} IxI$ where $\widetilde W$ is the extended affine Weyl group of $G$. A fundamental alcove of a $\sigma$-conjugacy class $[b]$ is an element $x_b\in \widetilde W$ such that $Ix_bI$ is $I$-$\sigma$-conjugate to a single element, which lies in $[b]$. For further properties of fundamental alcoves compare Nie \cite{Nie}. There, Nie also proves that for minuscule $\mu$ every $[b]\in B(G,\mu)$ has a fundamental alcove that in addition is contained in $W\mu W$ where $W$ is the finite Weyl group of $G$ (see \cite{Nie}, Prop. 1.5, or \cite{trunc1}, Thm 5.6 for a general existence theorem without boundedness by $\mu$). Fundamental alcoves for a given $[b]$ are in general not unique. However, if $G$ is split, all fundamental alcoves for a given $[b]$ lie in the same $K$-$\sigma$-conjugacy class (\cite{trunc1}, Lemma 6.11, or \cite{Nie}, Theorem 1.4). 

A special case of \cite{trunc1}, Prop. 5.7 yields
\begin{prop}\label{thmfund}
Let $\mu$ be minuscule, $[b]\in B(G,\mu)$, and let $x_b$ be a fundamental alcove of $[b]$ in $W\mu W$. Let $[b']\preceq [b]$. Then there is a fundamental alcove $x_{b'}$ in $W\mu W$ such that $\mathcal C_{x_{b'}}$ is contained in the closure of $\mathcal C_{x_b}$.\qed
\end{prop}

In Section \ref{sec3} we consider the first example of a group of type A where not all closures of central leaves for minuscule $\mu$ can be explained by the above lemma, which is the group $\GL_5$. Our main result for this example is

\begin{thm}\label{theorem_main}
Let $G=\GL_5$, $T\subset B$ the diagonal torus and upper triangular Borel, and let $\mu\in X_*(T)_{\dom}$ be minuscule. Let $\mathcal C_b$ be the central leaf of any $b\in K\e^{\mu} K$. Then $\overline{\mathcal C_b}$ contains every central leaf $\mathcal C_{x}$ where $x$ is a fundamental alcove in $W\mu W$ of some $[b']\prec [b]$.
\end{thm}

This theorem shows (to our surprise, compare Remark \ref{remexpect} for a more detailed comment) that closures of central leaves seem to not give a good correspondence between points in the corresponding affine Deligne-Lusztig varieties. It is rather the case that also from the point of view of these closure relations the fundamental alcoves play a very particular role in the sense that their central leaves are contained in the closure of all central leaves for larger $\sigma$-conjugacy classes.

\section{The global structure of central leaves in loop groups}\label{sec2}

\begin{definition} Let $\mathcal{B}$ be a subscheme of the loop group $LG$.
\begin{enumerate}
\item Let $\mu\in X_*(T)_{\dom}$. Then $\mathcal B$ is \emph{bounded by $\mu$} if it is contained in the closure of  $K\epsilon^\mu K$. It is \emph{bounded} if it is contained in a finite union of double cosets $K\epsilon^\mu K$.
\item Let $K_n$ be the kernel of the projection map $K\rightarrow G(\O_F/(t^n))$. Then $\mathcal B$ is \emph{admissible} if there is an $n\in \mathbb{N}$ with $\mathcal{B}K_n=\mathcal{B}$.
\item For a bounded and admissible algebraic set with $XK_n=X$ let $$\dim X:=\dim(X/K_n) -n\cdot\dim(G).$$
\end{enumerate}
\end{definition}
\begin{remark}\label{rem221}
\begin{enumerate}
\item Let $\mathcal{B}$ be bounded. Then one easily sees that $\mathcal{B}$ is admissible iff there is an $n'\in \mathbb{N}$ with $K_{n'}\mathcal{B}=\mathcal{B}$. Here $n'$ can be given in terms of the bound for $\mathcal B$ and the $n$ in the definition of admissibility.
\item Let $\mathcal{B}$ be bounded and admissible. Then $\mathcal{B}$ is \emph{smooth, locally closed, closed, irreducible etc.}~iff for some $n$ as above the quotient $\mathcal B/K_n\subset LG/K_n$ has the corresponding property.
\item The dimension of a bounded and admissible subscheme of $LG$ is independent of the choice of $n.$ 
\item Similarly, one can define the codimension of a closed irreducible admissible subscheme $\mathcal B'$ of some bounded and admissible scheme $\mathcal B$. If $\mathcal B$ is also equidimensional, one easily sees that this codimension agrees with $\dim\mathcal B-\dim \mathcal B'$.
\end{enumerate}
\end{remark}

\begin{prop}\label{prop23}
 Let $\mathcal{B}$ be a bounded subset of $LG(k)$. Then there is a $c\in \mathbb{N}$ such that for each $d\in \mathbb{N}$, each $g\in \mathcal{B}$ and $h\in K_{d+c}(k)$ there is an $l\in K_d(k)$ with $gh=l^{-1}g\sigma^*(l).$
\end{prop}
The proof is along the same lines as the proof of \cite{HV1} Theorem 10.1 where the same statement is shown for split groups. Thus we mainly indicate what changes one has to make to generalize to our situation. We also explain how to replace part of the proof by a reference to the main theorem of \cite{RZindag} and to the theory of fundamental alcoves. The last claim we give below is a better estimate than the corresponding claim in \cite{HV1} Lemma 10.2, due to the fact that in the meantime the theory of fundamental alcoves has been developed. One could also use essentially the same claim and proof as in \cite{HV1}.
\begin{proof}
We write $\mathcal B$ as a disjoint union of its intersections with the different $K$-double cosets and Newton strata. As by boundedness of $\mathcal B$ only finitely many of these intersections are non-empty, we may consider each of them separately and thus assume that $\mathcal B=\mathcal N_{[b],\mu}$ for some $\mu$ and $[b]\in B(G,\mu)$.

Let $x$ be a fundamental alcove for $[b]$. Then Rapoport and Zink \cite{RZindag} show that there is a bounded subset $\mathcal C$ of $LG$ such that each element of $\mathcal B$ is $\sigma$-conjugate via an element of $\mathcal C$ to (a fixed chosen representative in $LG$ of) $x$.

The same estimates as in the last paragraph of the proof of \cite{HV1} Theorem 10.1 then show that it is enough to prove the following claim: For every $d\in\mathbb N$ and $g\in I_d$ there is a $k\in I_d$ with $xg=k^{-1}x\sigma(k)$. Here $I_d$ is the subgroup of $K_d$ of elements which reduce to $1$ modulo $\e^d$ and whose reduction modulo $\epsilon^{d+1}$ is in $B$.

This claim in its turn is shown by the same arguments as the claim for $d=0$. By \cite{Nie}, Thm 1.3 every fundamental alcove is a so-called $P$-fundamental alcove for some $P$. For $P$-fundamental alcoves one can then follow the proof of \cite{GHKR2}, Proposition 6.3.1 to also prove our claim. One needs two replacements: Their statement is the claim for $d=0$, and they assume that $G$ is split, thus in our case we cannot assume $x$ or $P$ to be fixed by $\sigma$. Both lead to some obvious modifications that one has to make.  
\end{proof}

\begin{cor}
Let $b\in K\e^{\mu} K$. Then the $K$-$\sigma$-conjugacy class $\mathcal C_b$ is contained in $K\e^{\mu} K$, admissible, and a smooth and locally closed subscheme of $LG$. Further, $\N_{[b],\mu}$ is admissible.
\end{cor}
\begin{proof}
In both cases, admissibility follows from the previous proposition. As $\mathcal C_b$ is one $K$-orbit, it is smooth and locally closed.
\end{proof}
The assertion on Newton strata also follows from a corresponding assertion on Newton strata in the whole loop group by He \cite{HeKR}, Thm. A.1.

Our next goal is to compute the dimension of central leaves. We do so by relating them to analogously defined leaves in suitable deformation spaces, which can then be computed explicitly.

\begin{definition}
Let $n\in\mathbb N$ and let $b\in LG(k)$ be bounded by some $\mu\in X_*(T)_{\dom}$. We consider the following functor on the category $(Art/k)$ of Artinian local $k$-algebras with residue field $k$.
\begin{align*}
\mathscr{D}ef(b)_n:(Art/k)\to{}&(Sets),\\
  A\mapsto{}&\{\tilde b\in (\overline{K\e^{\mu}K})(A) \text{ with }\tilde b_k=b\}/_{\cong_n}.
\end{align*}
Here $\tilde b\cong_n \tilde b'$ if there exists a $g\in K_n(A)$ with $g_k=1$ such that $\tilde b'=g^{-1}\tilde b\sigma(g).$
We call $\mathscr{D}ef(b)_n$ the deformation functor of level $n$ of $b$.
\end{definition}

\begin{prop}\label{propdef}
The functor $\mathscr{D}ef(b)_n$ is pro-represented by the formal completion of $K_n\backslash\overline{K\e^{\mu} K}$ at $K_nb$, which we denote by $D_{b,n}$. 
\end{prop}
Here $\overline{K\e^{\mu} K}=\bigcup_{\mu'\preceq\mu}K\e^{\mu'}K$ denotes the closure of $K\e^{\mu}K$ in $LG$.

\begin{proof}
We follow the proof of the corresponding statement for $n=0$ and split $G$ in \cite{HV1}, Theorem 5.6. Let $p:\overline{K\e^{\mu} K}\to K_n\backslash \overline{K\e^{\mu} K}$ be the projection morphism. Let $D_{b,n}$ be the formal completion of $K_n\backslash\overline{K\e^{\mu} K}$ at $K_nb$. By \cite{foliat} Lemma 2.3 (or rather its generalization to unramified, not necessarily split $G$, which holds by the same proof), we have a section $b_n:D_{b,n}\to \overline{K\e^{\mu} K}\hookrightarrow LG$ of $p$ over $D_{b,n}$ which maps the closed point to $b$. Composition with $b_n$ yields a map of functors $D_{b,n}\to \mathscr{D}ef(b)_n$. We want to show that it is bijective. Let $\tilde b\in LG(A)$ be a representative of a class in $\mathscr{D}ef(b)_n(A)$ for some $A$ as above. We have to show that there is a unique homomorphism $u:\Spec A\rightarrow D_{b,n}$ such that $\tilde b\cong_n b_n\circ u$, which also does not depend on the choice of $\tilde b$ within its class in $\mathscr{D}ef(b)_n(A)$.

Let $\m\subset A$ be the maximal ideal. We use induction on $i$ to show that the same is true over $A_i=A/\m^{q^i}$, for some homomorphism $u_i:\Spec A_i\rightarrow D_{b,n}$ with $u_i\equiv u_{i-1}\pmod{\m^{q^{i-1}}}$ and $g_{i}^{-1}\tilde b\sigma(g_{i})=b_n\circ u_{i}\pmod{\m^{q^{i}}}$ for some $g_{i}\in K_n(A_{i})$ with $g_i\equiv g_{i-1}\pmod{\m^{q^{i-1}}}$ and $g_0=1$. Furthermore, we want to show that $u_i$ does not depend on the choice of $\tilde b$ within its class in $\mathscr{D}ef(b)_n(A)$.

For $i=0$ this is true by our choice of $b_n$. Let $\tilde b_i$ denote the restriction of $\tilde b$ to $\Spec A_i$. Assume that we have constructed $u_{i-1}:\Spec A_{i-1}\rightarrow D_{b,n}$ and $g_{i-1}\in K_n(A_{i-1})$ with $g_{i-1}^{-1}\tilde b_{i-1}\sigma(g_{i-1})= b_n\circ u_{i-1}\pmod{\m^{q^{i-1}}}$. As $K_n$ is smooth, we can choose a lift $g_i'\in K_n(A_i)$ of $g_{i-1}$. The morphism $({g_i}')^{-1}\tilde b_{i}\sigma(g_{i}')$ factors through the formal completion of $\overline{K\e^{\mu}K}$ in $b$. We thus obtain a morphism $u_i: \Spec A_i\rightarrow D_{b,n}$ with $u_i|_{\Spec A_{i-1}}=u_{i-1}$ and $\delta (g_i')^{-1}\tilde b_{i}\sigma(g_{i}')=b_n\circ u_i$ for some $\delta\in K_n(A_i).$ As $\delta|_{\Spec A_{i-1}}=1$, we have $\sigma(\delta)=1$ and thus $\delta (g_i')^{-1}\tilde b_{i}\sigma(g_{i}')=(g_{i}'\delta^{-1})^{-1} \tilde b_{i}\sigma(g_{i}'\delta^{-1})$ have the same class in $\mathscr{D}ef(b)_n(A_i)$. Thus $u_i$ and $g_i=g_{i}'\delta^{-1}$ are as required. The same proof also shows uniqueness of the $u_i$, and independence of the choice of $\tilde b$.
\end{proof}

We have a natural projection morphism $D_{b,n}\rightarrow D_{b,0}$ for every $n$. Using the universal object $b_0$ over $D_{b,0}$ we can define a section $s:D_{b,0}\rightarrow D_{b,n}$ which of course depends on the choice of the universal object. Let $(D_{b,0}\times (K_n\backslash K))^\wedge$ denote the completion of $D_{b,0}\times (K_n\backslash K)$ at $(b,1)$. Then $s$ induces a morphism $$\phi:(D_{b,0}\times (K_n\backslash K))^\wedge\to D_{b,n}$$ given
by $g^{-1}b_n(s)\sigma(g)$. Here $g$ is a local section of the projection $K\to K_n\backslash K$
at $1$ and $b_n$ is the universal object $D_{b,n}\to LG.$ Note that $g^{-1}b_n(s)\sigma(g)$ yields an element of $D_{b,n}$ independent of the choice of $b_n$. Indeed, different choices of $b_n$ lead to $b_n(s)$ differing by $K_n$-$\sigma$-conjugation and $K_n$ is normal in $K$. Furthermore, the morphism $\phi$ is independent the choice of the local section $g.$ It does, however, depend on the choice of $b_0$.

\begin{lemma}\label{lem27}
 The morphism $\phi:(D_{b,0}\times (K_n\backslash K))^\wedge\to D_{b,n}$ given above is an isomorphism.
\end{lemma}
\begin{proof} We construct an inverse of $\phi.$ Let $\tilde b\in D_{b,n}(A)$ for some $A$. From the natural projection $pr:D_{b,n}\to D_{b,0},$ we obtain a $g\in K(A)$ such that $\tilde b=g(b_0\circ pr(\tilde b))\sigma(g^{-1})$ and $g|_k=1.$ This $g$ is uniquely defined as an element of the completion at $1$ of $K_n\backslash K$. Thus the $g$ for varying $A$ glue to give a morphism $D_{b,n}\to (K_n\backslash K)_1^{\wedge}.$ Together with $pr$ we obtain a morphism $\psi:D_{b,n}\rightarrow (D_{b,0}\times (K_n\backslash K))^\wedge$. It is easy to see that $\psi$ is an inverse of $\phi$.
\end{proof}

\begin{lemma}
Let $R$ be an admissible $\mathbb{F}_q$-algebra with filtered index poset $\mathbb{N}_0$. Then pullback by the natural morphism $\Spf R\rightarrow \Spec R$ induces a bijection between $\Spec R$-valued points and $\Spf R$-valued points of $\overline{K\mu K}$.
\end{lemma}

\begin{proof}
This is the direct analog of \cite{HV1}, Prop. 3.16 (except that in the language of loc. cit. we are only considering trivial local $G$-shtukas here). It follows from the second half of the proof of loc.~cit.
\end{proof}

Using the lemma we can associate with the formal scheme $D_{b,n}$ a scheme $D_{b,n}'$, and the universal object induces a morphism  $D_{b,n}'\rightarrow LG$. In particular we can study the Newton stratification and central leaves on $D_{b,n}'.$

We now generalize \cite{foliat},Theorem 6.5 from the case that $G$ is split to our case of unramified groups $G$. The proofs are almost the same in this more general context. Therefore, we limit the proof we give here to only giving the references for new results we need from other papers, and to explaining the modifications that one needs to make.

We need the following notation from \cite{trunc1} Def. 6.1. Let $x$ be a fundamental alcove. We define $\phi_x:LG\rightarrow LG$ with $g\mapsto \sigma(xgx^{-1})$. By \cite{Nie}, Theorem 1.3 every fundamental alcove is $P$-fundamental for some semistandard parabolic subgroup $P$. Let $\overline N$ be the unipotent radical of the opposite parabolic and let $I_{\overline N}=I\cap L\overline N$. Then by definition of $P$-fundamental alcoves we have $\phi_x(I_{\overline N})\supseteq I_{\overline N}$, i.e. $\phi_x^{-1}(I_{\overline N})=x^{-1}I_{\sigma^{-1}(\bar{N})}x\subseteq I_{\bar{N}}$.

\begin{thm}\label{thmfoliat}
 Let ${b}\in K\e^{\mu} K$. Let $\mathcal{N}_{[b]}$ be the Newton stratum of $[b]$ in
 $\Spec D_{b,0}'.$ Let $x$ be a $P$-fundamental alcove associated with $[b]$ where $P$ is chosen such that the Levi subgroup $M$ of $P$ equals the centralizer of the $M$-dominant Newton point of $x.$ Then there is a reduced scheme $S$ and a finite surjective morphism $S\to \mathcal{N}_{[b]}$ which factors into finite surjective morphisms $$S\to ((X_{\leq\mu,K}(b)\times_k I_{\bar{N}}/ x^{-1}I_{\sigma^{-1}(\bar{N})}x)^\wedge)' \to \mathcal{N}_{[b]}.$$ Here again, the $(\cdot)'$ denotes the scheme associated with the formal scheme obtained by completion.
 Furthermore, the central leaf $\mathcal C_{b}$ of $b$ in $\Spec D_{b,0}'$ is smooth and equal to the image of $(\{1\}\widehat{\times}_k I_{\bar{N}}/ x^{-1}I_{\sigma^{-1}(\bar{N})}x)^\wedge)'$ in $\mathcal{N}_{[b]}.$
\end{thm}

\begin{proof}
This follows from essentially the same proof as \cite{foliat}, Theorem 6.5, making the following modification. As the fundamental alcove and the associated parabolic subgroup $P$ are in general no longer fixed by $\sigma$, one has to replace conjugation by $x$ with the morphism $\phi_x$ defined above. For example, the frequently occurring subgroups $x^{-n}I_{\bar N} x^n$ are replaced by $\phi_x^{-n}(I_{\bar N})$. To make sense of the right hand side we use that $\sigma$ acts on the root system and fixes $I$, so $\sigma^{-1}(I_{\bar N})=I_{\sigma^{-1}(\bar N)}$ where $\sigma^{-1}(\bar N)$ is again the unipotent radical of a semistandard parabolic subgroup.
\end{proof}

Together with Lemma \ref{lem27} this implies the following corollary.
\begin{cor}
 Let $b\in LG(k)$ be bounded by a dominant $\mu$ and let $n\in\mathbb N$. Let $\mathcal{N}_{[b],n}$ be the Newton stratum of $[b]$ in
 $D_{b,n}'.$ Let $x$ be a $P$-fundamental alcove in $[b]$ such that the Levi subgroup $M$ of $P$ equals the centralizer of the $M$-dominant Newton point of $x.$ Then there is a reduced scheme $S$ and a finite surjective morphism $S\to \mathcal{N}_{[b],n}$ which factors into finite surjective morphisms $$S\to (X_{\leq\mu,K}(b)\times_kK/K_n\times_k (I_{\bar{N}}/ x^{-1}I_{\sigma^{-1}(\bar{N})}x)^\wedge)' \to \mathcal{N}_{[b],n}.$$
 Furthermore, $\mathcal C_{b,\mathcal{N}_{[b],n}}$ is smooth and equal to the image of $(\{1\}\times_kK/K_n\times_k I_{\bar{N}}/ x^{-1}I_{\sigma^{-1}(\bar{N})}x)^\wedge)'$ in $\mathcal{N}_{[b],n}.$
\end{cor}

\begin{thm}\label{thmdim}
Let $\mu\in X_*(T)_{\dom}$  and $[b]\in B(G,\mu)$ with fundamental alcove $x$. For $y\in \mathcal N_{[b],\mu}$ we denote again by $\mathcal C_y$ the $K$-$\sigma$-conjugacy class of $y$. Then $$\dim \mathcal C_y=\ell(x),$$ only depending on $[b]$, and not on $y$ or $\mu$.
\end{thm}
Note that as $x$ is fundamental for $[b]$ we have $\ell(x)=\langle 2\rho,\nu\rangle$ where $\nu$ is the Newton point of $[b]$.
\begin{proof}
Let $n$ be big enough such that all central leaves in $K\e^{\mu} K$ are left and right $K_n$-invariant, cf. Proposition \ref{prop23}. Then the codimension of $\mathcal C_y$ in $K\e^{\mu} K$ coincides with the codimension of its image in $K_n\backslash K\e^{\mu} K$. Let $\mathcal C_{y,D_{y,n}}\subseteq D_{y,n}$ be the pullback of the leaf of $y$ in $D_{y,n}'$ to $D_{y,n}$. Going through the proof of Proposition \ref{propdef} (and using our choice of $n$) one sees that the completion of this image in $y$ coincides with $\mathcal C_{y,D_{y,n}}$ under the isomorphism of $D_{y,n}$ with the completion of $K_n\backslash K\e^{\mu} K$ in $y$. Hence the codimension of $\mathcal C_y\subset K\e^{\mu} K$ is equal to the codimension of the leaf of $y$ in $D_{y,n}'$. Furthermore, the dimension of the ambient schemes $K\e^{\mu} K$ and $D_{y,n}$ are $\ell(\mu)$ resp. $\ell(\mu)+\dim K/K_n$ (for the respective notions of dimension). Thus $\dim\mathcal C_y=\dim\mathcal C_{y,D_{y,n}}-\dim K/K_n$. By the preceeding corollary $\dim\mathcal C_{y,D_{y,n}}=\dim (I_{\bar{N}}/x^{-1}I_{\bar{N}}x))+\dim K/K_n$, hence $\dim\mathcal C_y=\dim (I_{\bar{N}}/xI_{\sigma^{-1}(\bar{N})}x))=\ell(x)$.
\end{proof}

\begin{cor}\label{corclosed}
For $y\in \mathcal{N}_{[b],\mu}$, the central leaf $\mathcal C_y$ is closed in $\mathcal{N}_{[b],\mu}$.
\end{cor}
\begin{proof}
As $\mathcal C_y$ is bounded and admissible, the same holds for its closure in $\mathcal{N}_{[b],\mu}$. The closure is invariant under $K$-$\sigma$-conjugation, hence it is a union of central leaves. Thus it has to be a union of $\mathcal C_y$ with central leaves of strictly smaller dimension. But as the dimension of all central leaves in this Newton stratum is the same, $\mathcal C_y$ has to be closed in $\mathcal{N}_{[b],\mu}$.
\end{proof}

\section{The minuscule case for $\GL_5$}\label{sec3}
In this section we prove Theorem \ref{theorem_main}. Let $G=\GL_5$, $B$ the Borel subgroup of upper triangular matrices and $T$ the diagonal torus.
\subsection{Reduction to two particular cases}\label{sec31}

Let $\mu$ be minuscule. Identifying $X_*(T)$ with $\mathbb{Z}^5$, this implies that $\mu\in X_*(T)_{\dom}$ is of the form $(i,\dotsc,i,i-1,\dotsc,i-1)$ for some $i\in \Z$ and with multiplicities $n$ and $5-n$ for some $n$. Modifying $\mu$ by a central cocharacter does not change our question. Thus we may assume that $i=1$.

A direct calculation using the dimension formula of \cite{GHKR} and \cite{dimdlv} shows that $\dim X_{\mu}(b)=0$ for all pairs $([b],\mu)$ unless $n=2$ or $3$. Thus the theorem holds for $n=0,1,4,5$ by Lemma \ref{lem1}. 

The cases $n=2$ or $3$ are analogous to each other as $(1,1,1,0,0)$ differs from $(1,1,0,0,0)$ by a central cocharacter together with inverting $\mu$. Thus from now on we assume that $n=2$. For this case, the set $B(G,\mu)$ has eight elements $[b_i]$, with the following Newton points $\nu_{i}\in\Q^5$.\\[2mm]
\begin{minipage}[t]{0.5\textwidth}

\begin{enumerate}
\item[] $\nu_{1}=\big(\frac{2}{5}^{(5)}\big)$\\
\item[] $\nu_{2}=\big(\frac{1}{2}^{(2)},\frac{1}{3}^{(3)}\big)$\\
\item[] $\nu_{3}=\big(\frac{1}{2}^{(4)},0\big)$\\
\item[] $\nu_{4}=\big(1,\frac{1}{4}^{(4)}\big)$
\end{enumerate}

\end{minipage}
\begin{minipage}[t]{0.5\textwidth}
\begin{enumerate}
\item[] $\nu_{5}=\big(\frac{2}{3}^{(3)},0^{(2)}\big)$\\
\item[] $\nu_{6}=\big(1,\frac{1}{3}^{(3)},0\big)$\\
\item[] $\nu_{7}=\big(1,\frac{1}{2}^{(2)},0^{(2)}\big)$\\
\item[] $\nu_{8}=\big(1^{(2)},0^{(3)}\big)$.
\end{enumerate}

\end{minipage}\\[2mm]

\noindent Here exponents in brackets denote the multiplicity of the corresponding entry.

The dimension formula of \cite{GHKR} and \cite{dimdlv} implies $$\dim X_{\mu}(b_i)=\begin{cases}1&\text{if }i\leq 3\\
0&\text{otherwise.}
\end{cases}$$
Thus, Lemma \ref{lem1} shows that the theorem is true for every central leaf in a Newton stratum for $[b_i]$ with $i\geq 4$. The Newton stratum for $\nu_1$ is the unique closed Newton stratum. Thus there is nothing to show in that case. Using that the ordering between $[b_1],[b_2],[b_3]$ is total and $[b_1]\prec[b_2]\prec [b_3]$, we see that it remains to show the following two assertions.

\begin{lemma}\label{lemma2}
In the above situation let $\mathcal C_b$ be a central leaf for any $b\in [b_3]\cap K\e^{\mu} K$. Then $\overline{\mathcal C_b}$ contains the central leaf $\mathcal C_{x}$ where $x$ is a fundamental alcove in $W\mu W$ of $[b_2]$.
\end{lemma}
\begin{lemma}\label{lemma3}
In the above situation let $\mathcal C_b$ be a central leaf for any $b\in [b_2]\cap K\e^{\mu} K$. Then $\overline{\mathcal C_b}$ contains the central leaf $\mathcal C_{x}$ where $x$ is a fundamental alcove in $W\mu W$ of $[b_1]$.
\end{lemma}
Recall that as $G$ is split, the central leaf $\mathcal C_{x}$ is uniquely determined by $[b_i]$. 

The proof of these lemmas is the content of the following two subsections.

\begin{remark}\label{remexpect}
Note that for $[b_i]$ with $i\leq 3$ one can show that $J_{b_i}(F)$ is acting transitively on the set of irreducible components of $X_{\mu}(b_i)$ and that each irreducible component is isomorphic to $\mathbb{P}^1$ in each of these cases.

Furthermore, the dimensions of the central leaves for $b_i,b_{i-1}$ differ by 1, so the closure of $\mathcal C_{b_i}$ contains a finite number of leaves of $[b_{i-1}]$. Naively, one might hope that these two observations are related in the sense that the closure of $\mathcal C_{b_i}$ contains exactly one leaf in $[b_{i-1}]$, and that in this way we obtain a bijection (or at least a finite correspondence) between the irreducible components of affine Deligne-Lusztig varieties. However, Theorem \ref{theorem_main} shows that this expectation is not correct. In fact quite the contrary is true in the sense that there is one leaf for $[b_{i-1}]$ (for the fundamental alcove) that is in the closure of \emph{every} leaf of $[b_i]$.
\end{remark}

\subsection{Proof of Lemma \ref{lemma2}}

Notice that in this case the two relevant Newton polygons have the slope $1/2$ (with multiplicity 2) in common. Thus as a preparation for the proof we will first consider the situation where $G=\GL_3$ and $\mu'=(1,0,0)$, $[b_2']$ the $\sigma$-conjugacy class of slope $1/3$ and $[b_3']$ of slopes $(1/2,1/2,0)$. Then $X_{\mu'}(b_i')$ is 0-dimensional (for $b_i'$ any representative of $[b_i']$ and $i=2,3$). By Lemma \ref{lem1} this implies that the Newton strata $\mathcal N_{\mu',[b_i']}$ consist of a single central leaf, and their closure is known. However, for later use we need more explicit information.   

Let $k'= \overline{k\dl \pi\dr }$ be an algebraic closure of $k\dl \pi\dr $ for a variable $\pi$. Computing Newton polygons we see that
 $$\left(\begin{array}{ccc}
 0& 1& 0 \\
 0& \pi& 1 \\
 \epsilon& 0& 0
 \end{array}\right)\text{ and }
 \left(\begin{array}{ccc}
 1& 0& 0 \\
 0& 0& 1 \\
 0& \epsilon& 0
 \end{array}\right)
 $$ are in the central leaf for $[b_3']$, and hence $G(k'\ll\epsilon\rr )$-$\sigma$-conjugate. In other words, there exists a $g\in G(k'\ll\epsilon\rr )$ such that
\begin{equation}\label{eq1}
\left(\begin{array}{ccc}
 0& 1& 0 \\
 0& \pi& 1 \\
 \epsilon& 0& 0
 \end{array}\right)\sigma(g)=
 g \left(\begin{array}{ccc}
 1& 0& 0 \\
 0& 0& 1 \\
 0& \epsilon& 0
 \end{array}\right).
\end{equation}
 Let us calculate $g$ explicitly. Let $g=\sum_{i=0}^\infty \epsilon^i(g_{jk}^i)_{1\leq j,k\leq 3}.$ Then \eqref{eq1} reads
 $$\sum_{i=0}^\infty \epsilon^i \left(\begin{array}{ccc}\sigma(g_{21}^i)& \sigma(g_{22}^i)& \sigma(g_{23}^i) \\
 \pi\sigma(g_{21}^i)+\sigma(g_{31}^i)& \pi\sigma(g_{22}^i)+\sigma(g_{32}^i)& \pi\sigma(g_{23}^i)+\sigma(g_{33}^i) \\
 \sigma(g_{11}^{i-1})&\sigma(g_{12}^{i-1})& \sigma(g_{13}^{i-1})
 \end{array}\right)=\sum_{i=0}^\infty \epsilon^i \left(\begin{array}{ccc}g_{11}^i& g_{13}^{i-1}& g_{12}^i \\
 g_{21}^i& g_{23}^{i-1}& g_{22}^i \\
 g_{31}^i&g_{33}^{i-1}& g_{32}^i
 \end{array}\right) $$ where we put $g^{-1}_{jk}=0$ for all $j,k$. We solve these equations by comparing the summands on the two sides and using induction on $i$. We will express all variables $g^i_{jk}$ in terms of those for $(j,k)=(2,1)$, $(1,3)$ and $(2,3)$ and then give the relations for those variables. In the following we refer to the equation we get from comparing the $(j,k)$th entries at step $i$ by $(jk^i)$. We obtain
\begin{align}
g^i_{11}={}&\sigma(g^i_{21})\tag{$11^i$}\\
g^i_{31}={}&\sigma(g_{11}^{i-1})=\sigma^2(g_{21}^{i-1}) \tag{$31^i$}\\
g^i_{21}={}&\pi\sigma(g^i_{21})+\sigma(g^i_{31})= \pi\sigma(g^i_{21})+\sigma^3(g_{21}^{i-1})\tag{$21^i$}\\
\sigma(g^i_{22})={}&g^{i-1}_{13}\tag{$12^i$}\\
\sigma(g^i_{32})={}&g^{i-1}_{23}-\pi \sigma(g^i_{22})=g^{i-1}_{23}-\pi g^{i-1}_{13}\tag{$22^i$}\\
g^i_{12}={}&\sigma(g^i_{23})\tag{$13^i$}\\
g^i_{33}={}&\sigma(g^i_{12})=\sigma^2(g_{23}^i)\tag{$32^{i+1}$}\\
g^i_{22}={}&\pi\sigma(g^i_{23})+\sigma(g^i_{33})= \pi\sigma(g_{23}^i)+\sigma^3(g_{23}^i)\tag{$23^i$}\\
g^i_{32}={}&\sigma(g^{i-1}_{13}).\tag{$33^i$}
\end{align} 
From ($12^{i}$) and ($23^{i}$) resp. from ($22^{i}$) and ($33^{i}$) we obtain 
\begin{align}
\sigma(\pi)\sigma^2(g_{23}^i)+\sigma^4(g_{23}^i)={}& g_{13}^{i-1}\tag{$*^i$}\\
\sigma^2(g_{13}^i)+\pi g_{13}^i={}&g_{23}^i.\tag{$**^i$}
\end{align}
Thus all of these equations can be solved: The variables $g_{23}^i$, $g_{13}^i$ satisfy equations ($*^i$) and ($**^i$), all others satisfy one of the recursive equations ($jk^i$) above. For example, for $i=0$ we obtain
$$ g^0=(g^0_{jk})=\left(\begin{array}{ccc}
 \sigma(g^0_{21})& \sigma(g_{23}^0)& g_{13}^0 \\
 g^0_{21}& 0& g_{23}^0 \\
 0& 0& \sigma^2(g_{23}^0)
 \end{array}\right).$$ Here $g_{23}^0$ is a solution of $\pi\sigma(g_{23}^0)+ \sigma^3(g_{23}^0)=0$. We choose it to be non-zero, as this condition is needed for $g^0$ to be invertible. In other words, $g_{23}^0$ is a $q^3-q$th root of $-\pi$. Further, $g_{21}^0$ is a non-zero solution of ($21^0$), i.e.~a $q-1$st root of $1/\pi$. Finally, $g_{13}^0$ satisfies ($**^0$). One can then easily check that $g^0$ is invertible. In particular, the element $g$ defined in this way is indeed in $G(k'\ll\epsilon\rr ).$ 

Consider on $\overline{k\dl \pi\dr }$ the valuation defined by $\pi$, and let $R$ be the valuation ring. Using induction and ($*^i$) and ($**^i$), one obtains $v(g_{23}^i)=\frac{1}{q^{6i+1}(q^2-1)}$ and $v(g_{13}^i)=\frac{1}{q^{6i+3}(q^2-1)}$.

We now return to the situation where $G=\GL_5$. 

\begin{lemma}\label{lemrep1}
Every central leaf in $\mathcal N_{[b_3],\mu}$ has a representative of the form $$x_t=\left(\begin{array}{ccccc}
 1& 0& 0& 0& 0\\
 0& 0& 1& t& 0\\
 0& \epsilon& 0& 0& 0\\
 0& 0& 0& 0& 1\\
 0& 0& 0& \epsilon& 0
 \end{array}\right).$$
\end{lemma}
\begin{proof}
We have to show that every $x\in X_{\mu}(b_3)$ has a representative $g$ such that $g^{-1}b_3\sigma(g)$ is of the above form. By the Hodge-Newton decomposition \cite{Katz}, one first shows that there always is a representative with $g$ in the Levi subgroup $\mathbb G_m\times \GL_4$. The remaining argument is a standard calculation, for example along the same lines as in \cite{Kaiser}, 3 where the analogous case for $G=\GSp_4$ is considered, or in \cite{modpdiv}, where a generic subscheme of any affine Deligne-Lusztig variety for $\GL_n$ and minuscule $\mu$ is described. We therefore leave the details to the reader.
\end{proof}
\begin{remark}
A representative of the central leaf of a fundamental alcove of $[b_2]$ is given by the matrix
$$b_2:=\left(\begin{array}{ccccc}
 0& 1& 0& 0& 0\\
 0& 0& 1& 0& 0\\
 \epsilon& 0& 0& 0& 0\\
 0& 0& 0& 0& 1\\
 0& 0& 0& \epsilon& 0
 \end{array}\right),$$ as can be seen from an elementary calculation or using \cite{foliat}, Example 4.4.
\end{remark}

 \begin{proof}[Proof of Lemma \ref{lemma2}]
By the above lemma and remark it remains to show that for every $t$, the element $b_2$ is contained in the closure of $\mathcal{C}_{x_t}$.

We use the element $g\in \GL_3(k'\ll\epsilon\rr )$ constructed above, and compute
$$ \left(\begin{array}{ccc}
g& & \\
& 1& 0\\
& 0& 1
\end{array}\right)
\left(\begin{array}{ccccc}
1& 0& 0& 0& 0\\
0& 0& 1& t& 0\\
0& \epsilon& 0& 0& 0\\
0& 0& 0& 0& 1\\
0& 0& 0& \epsilon& 0
\end{array}\right)
\left(\begin{array}{ccc}
\sigma(g^{-1})& & \\
& 1& 0\\
& 0& 1
 \end{array}\right)=\left(\begin{array}{ccccc}
 0& 1& 0& t_1& 0\\
 0& \pi& 1& t_2& 0\\
 \epsilon& 0& 0& t_3& 0\\
 0& 0& 0& 0& 1\\
 0& 0& 0& \epsilon& 0
 \end{array}\right).$$
 Here
 $$\left(\begin{array}{c}
 t_1 \\
 t_2\\
 t_3
 \end{array}\right)=g\left(\begin{array}{c}
 0 \\
 t\\
 0
 \end{array}\right)=t\sum_{i\geq 0}\epsilon^i\left(\begin{array}{c}
 g_{12}^i \\
 g_{22}^{i}\\
 g_{32}^{i} \end{array}\right)
 =t\sum_{i\geq 0}\epsilon^i\left(\begin{array}{c}
 (g_{23}^i)^q\\
 (g_{13}^{i-1})^\frac{1}{q}\\
 (g_{13}^{i-1})^q
 \end{array}\right).
 $$
 As $v(g_{13}^i), v(g_{23}^i)>0$, this implies that
 $$\left(\begin{array}{ccccc}
 0& 1& 0& t_1& 0\\
 0& \pi& 1& t_2& 0\\
 \epsilon& 0& 0& t_3& 0\\
 0& 0& 0& 0& 1\\
 0& 0& 0& \epsilon& 0
 \end{array}\right)\in LG(R)
 $$
is an $R$-valued point of $\mathcal C_{x_t}$. Note that we could even consider $t$ as a variable, and in this way produce a point in $LG(R[t])$. Putting $\pi=0$ we get $b_2\in \overline{\mathcal C_{x_t}}.$
 \end{proof}

\subsection{Proof of Lemma \ref{lemma3}}

To prove Lemma \ref{lemma3} we proceed in a similar way as in the preceeding subsection. We choose representatives
$$b_2=\left(\begin{array}{ccccc}
 0& 1& 0& 0& 0\\
 0& 0& 1& 0& 0\\
 \epsilon& 0& 0& 0& 0\\
 0& 0& 0& 0& 1\\
 0& 0& 0& \epsilon& 0
 \end{array}\right),\quad b_1=\left(\begin{array}{ccccc}
 0& 0& 1& 0& 0\\
 0& 0& 0& 1& 0\\
 0& 0& 0& 0& 1\\
 \epsilon& 0& 0& 0& 0\\
 0& \epsilon& 0& 0&0
 \end{array}\right).$$
Then $b_1, b_2$ are representatives of the central leaf of the fundamental alcove for $[b_1]$ resp. of $[b_2]$.

\begin{lemma} For $k'=\overline{k\dl \pi\dr }$ as above we have
$$x_{\pi}=\left(\begin{array}{ccccc}
 0& 0& 1& 0& 0\\
 0& 0& 0& 1& 0\\
 0& 0& 0& \pi& 1\\
 \epsilon& 0& 0& 0& 0\\
 0& \epsilon& 0& 0& 0
 \end{array}\right)\in \mathcal C_{b_2}.$$
\end{lemma}
\begin{proof}
The element $b_2$ is $K$-$\sigma$-conjugate to the fundamental alcove of $[b_2]$. In particular, its truncation of level 1 is equal to the truncation of level 1 of that element. The definition of fundamental alcoves $y$ implies that the set of elements in the truncation stratum of $y$ agrees with the central leaf of $y$.  

For $x_{\pi}$ a direct calculation using the algorithm in the proof of \cite{trunc1}, Theorem 1.1 shows that it has the same truncation of level 1 as $b_2$. Hence also the corresponding central leaves agree.
\end{proof}
We need again some information about the explicit element $g\in G(k'\ll \epsilon\rr )$ with $g^{-1}b_2\sigma(g)=x_{\pi}$. Let $g=\sum\limits_{i=0}^\infty \epsilon^i(g_{jk}^i)_{1\leq j,k\leq 5}.$ Then 
$$b_2\sigma(g)=gx_{\pi}$$ is equivalent to
$$\left(\begin{array}{lllll}
 \sigma(g_{21}^i)& \sigma(g_{22}^i)& \sigma(g_{23}^i)& \sigma(g_{24}^i)& \sigma(g_{25}^i)\\
 \sigma(g_{31}^i)& \sigma(g_{32}^i)& \sigma(g_{33}^i)& \sigma(g_{34}^i) &\sigma(g_{35}^i)\\
 \sigma(g_{11}^{i-1})& \sigma(g_{12}^{i-1})& \sigma(g_{13}^{i-1})& \sigma(g_{14}^{i-1})& \sigma(g_{15}^{i-1})\\
 \sigma(g_{51}^i)& \sigma(g_{52}^i)& \sigma(g_{53}^i)& \sigma(g_{54}^i)& \sigma(g_{55}^i)\\
 \sigma(g_{41}^{i-1})& \sigma(g_{42}^{i-1})& \sigma(g_{43}^{i-1})& \sigma(g_{44}^{i-1})& \sigma(g_{45}^{i-1})
 \end{array}\right)=\left(\begin{array}{ccccc}
 g_{14}^{i-1}& g_{15}^{i-1}& g_{11}^i& g_{12}^i+\pi g_{13}^i& g_{13}^i\\
 g_{24}^{i-1}& g_{25}^{i-1}& g_{21}^i& g_{22}^i+\pi g_{23}^i& g_{23}^i\\
 g_{34}^{i-1}& g_{35}^{i-1}& g_{31}^i& g_{32}^i+\pi g_{33}^i& g_{33}^i\\
 g_{44}^{i-1}& g_{45}^{i-1}& g_{41}^i& g_{42}^i+\pi g_{43}^i& g_{43}^i\\
 g_{54}^{i-1}& g_{55}^{i-1}& g_{51}^i& g_{52}^i+\pi g_{53}^i& g_{53}^i
 \end{array}\right)$$ for every $i\geq 0$, again using the convention that entries with negative upper indices are $0$. Comparing coefficients and expressing everything in terms of the second column we get
$$(g^i_{jk})_{j,k}=\left(\begin{array}{lllll}
 {\sigma^3}(g_{12}^i)& g_{12}^i&{\sigma^2} (g_{32}^{i+1})& {\sigma^4}(g_{22}^{i+1})& \sigma(g_{22}^{i+1})\\
 {\sigma^3}(g_{22}^i)& g_{22}^i&{\sigma^2} (g_{12}^i)& {\sigma^4}(g_{32}^{i+1})& \sigma(g_{32}^{i+1})\\
 {\sigma^3}(g_{32}^i)& g_{32}^i&{\sigma^2} (g_{22}^i)& {\sigma^4}(g_{12}^i)& \sigma(g_{12}^i)\\
 {\sigma^3}(g_{52}^i)& g_{42}^i& {\sigma^2}(g_{42}^i)& {\sigma^4}(g_{42}^i)& \sigma(g_{52}^{i+1})\\
 {\sigma^3}(g_{42}^{i-1})& g_{52}^i& {\sigma^2}(g_{52}^i)& {\sigma^4}(g_{52}^i)& \sigma(g_{42}^i)
 \end{array}\right)$$
 satisfying
 \begin{align}
\label{eq2} {\sigma^5} (g_{12}^i)={}& g_{22}^i+\pi\sigma^2(g_{12}^i)\\
 {\sigma^5}(g_{22}^i)={}& g_{32}^i+\pi\sigma^2(g_{22}^i)\\ 
\label{eq4}{\sigma^5}(g_{32}^{i+1})={}& g_{12}^i+\pi\sigma^2(g_{32}^{i+1})\\
\nonumber{\sigma^5}(g_{42}^i)={}& g_{52}^{i+1}+\pi\sigma^2(g_{52}^{i+1})\\ 
\nonumber {\sigma^5}(g_{52}^i)={}&g_{42}^i+\pi{\sigma^2}(g_{42}^i).
 \end{align}
We consider again on $k'=\overline{k\dl \pi\dr }$ the valuation defined by $\pi$. Using induction and \eqref{eq2}--\eqref{eq4} we get $v(g_{12}^i)=\frac{1}{q^{15i+12}(q^3-1)}, v(g_{22}^i)=\frac{1}{q^{15i+7}(q^3-1)}, v(g_{32}^i)=\frac{1}{q^{15i+2}(q^3-1)}.$ 

Using an analogous calculation one can compute $g^{-1}$. We are mainly interested in the last two columns. Put $g^{-1}=\sum_{i=0}^\infty \epsilon^i(h_{jk}^i)_{1\leq j,k\leq 5}$. Then one obtains
 $$g^{-1}=\sum_{i=0}^\infty\epsilon^i\left(\begin{array}{ccccc} 
 *& *& *& \sigma(h_{35}^i) & \sigma(h_{34}^{i+1})\\
 *& *& *& \sigma^3(h_{35}^i) & \sigma^3(h_{34}^{i+1})\\
 *& *& *& h_{34}^i & h_{35}^i\\
 *& *& *& \sigma^2(h_{34}^i) & \sigma^2(h_{35}^i)\\
 *& *& *& \sigma^4(h_{34}^i) & \sigma^4(h_{35}^i)\\
 \end{array}\right)$$
 satisfying
\begin{align*}
h_{34}^i={}&\pi\sigma(h_{45}^i)+\sigma(h_{55}^i)=\pi \sigma^3(h_{35}^i)+\sigma^5(h_{35}^i)\\
h_{35}^{i-1}={}&\pi\sigma(h_{44}^i)+\sigma(h_{54}^i)=\pi \sigma^3(h_{34}^i)+\sigma^5(h_{34}^i).
\end{align*}
Using induction and these two relations one obtains  $v(h_{34}^i)=\frac{1}{q^{10i}(q^5-q^3)}, v(h_{35}^i)=\frac{1}{q^{10i+5}(q^5-q^3)}.$

\begin{lemma}
Every central leaf in $\mathcal N_{[b_2],\mu}$ has a representative of the form $$x_t=\left(\begin{array}{ccccc}
 0& 1& 0& 0& 0\\
 0& 0& 1& 0& 0\\
 \epsilon& 0& 0& 0& 0\\
 t& 0& 0& 0& 1\\
 0& 0& 0& \epsilon& 0
 \end{array}\right).$$
\end{lemma}
\begin{proof}
We have to show that every $x\in X_{\mu}(b_2)$ has a representative $g$ such that $g^{-1}b_2\sigma(g)$ is of the above form. As for Lemma \ref{lemrep1} this is a standard calculation (and very similar to the one in that lemma). We therefore omit the details. 
\end{proof} 
\begin{proof}[Proof of Lemma \ref{lemma3}]
By the previous lemma it is enough to show that $b_1$ is contained in the closure of the central leaf $\mathcal C_{x_t}$ for any given $t$. 

 We compute
 $$g^{-1}x_t\sigma(g)=\left(\begin{array}{ccccc}
 0& 0& 1& 0& 0\\
 0& 0& 0& 1& 0\\
 0& 0& 0& \pi& 1\\
 \epsilon& 0& 0& 0& 0\\
 0& \epsilon& 0& 0& 0
 \end{array}\right)+g^{-1}\left(\begin{array}{ccccc}
 0& 0& 0& 0& 0\\
 0& 0& 0& 0& 0\\
 0& 0& 0& 0& 0\\
 t& 0& 0& 0& 0\\
 0& 0& 0& 0& 0
 \end{array}\right)\sigma(g).$$
The entries of the second summand are products of $t$ with the entries of the fourth column of $g^{-1}$ and the first row of $\sigma(g)$. All of these entries are in $R\ll \epsilon\rr $ (compare the calculations of the corresponding valuations above). Hence this summand, and thus also the whole right hand side, are indeed in $LG(R)$. Setting $\pi=0$ we get $b_1\in \overline{\mathcal C_{x_t}}.$ Note that again, it would have been possible to do the same construction with a family parametrized by the variable $t$ instead of with a fixed $t$.
\end{proof}

\noindent{\it Acknowledgments.} We thank P.~Hamacher for helpful comments on a previous version. The second author thanks S.~Neupert and A.~Ivanov for helping him, both mathematically and with daily issues during his stay in Germany.

\end{document}